\documentclass[11pt]{amsart}
\PassOptionsToPackage{usenames,dvipsnames}{xcolor}
\usepackage{txfonts}
\PassOptionsToPackage{hyphens}{url}\usepackage{hyperref}
\usepackage{subfig}
\usepackage{amscd,amssymb,mathabx}
\usepackage[arrow,matrix,graph,frame,poly,arc,tips]{xy}
\usepackage{graphicx}
\usepackage{mathtools}
\usepackage{tikz}
\usetikzlibrary{decorations.markings}
\usetikzlibrary{shapes}
\usetikzlibrary{arrows}
\usetikzlibrary{backgrounds}
\usetikzlibrary{calc}
\usepackage{mdwlist}
\usepackage{multirow}
\usepackage{enumerate}
\usepackage[shortlabels]{enumitem}

\usepackage{verbatim}
\usepackage{etoolbox}
\AtEndEnvironment{proof}{\setcounter{claim}{0}}

\renewcommand{\comment}[1]{\marginpar{{\tiny{#1}\normalfont\par}}}

\newcommand\ba{\begin{align*}}
\newcommand\ea{\end{align*}}
\newcommand\be{\begin{enumerate}}
\newcommand\ee{\end{enumerate}}
\newcommand\bp{\begin{proof}}
\newcommand\ep{\end{proof}}
\newcommand\bpp{\begin{prop}}
\newcommand\epp{\end{prop}}
\newcommand\bpb{\begin{prob}}
\newcommand\epb{\end{prob}}
\newcommand\bd{\begin{defn}}
\newcommand\ed{\end{defn}}
\newcommand\bh{\begin{hint}}
\newcommand\eh{\end{hint}}

\newcommand\sgn{\mathrm{sgn}}

\newcommand\bN{\mathbb{N}}
\newcommand\N{\mathbb{N}}
\newcommand\bR{\mathbb{R}}
\newcommand\R{\mathbb{R}}

\newcommand\bZ{\mathbb{Z}}
\newcommand\Z{\mathbb{Z}}

\newcommand\Hom{\operatorname{Hom}}

\newcommand\supp{\operatorname{supp}}

\DeclareMathOperator\Homeo{Homeo}

\newcommand\sse{\subseteq}
\newcommand\co{\colon}

\DeclareMathOperator\Fix{Fix}

\DeclareMathOperator\Diff{Diff}

\usepackage[showseconds=false, showzone=false]{datetime2}

\def\thetitle{{Subexponential growth and $C^1$ actions on one-manifolds}}
\def\theauthors{{Sang-hyun Kim, Nicolás Matte Bon, Mikael de la Salle, Michele Triestino}}
\usepackage{hyperref}
\hypersetup{
   colorlinks=false,
   plainpages,
   urlcolor=black,
   linkcolor=black
   pdftitle=   \thetitle,
   pdfauthor=  {\theauthors}
}

\theoremstyle{plain}
\newtheorem{thm}{Theorem}[section]
\newtheorem{lem}[thm]{Lemma}
\newtheorem{cor}[thm]{Corollary}
\newtheorem{prop}[thm]{Proposition}

\newtheorem{que}[thm]{Question}
\newtheorem*{principle*}{Principle}
\newtheorem*{claim*}{Claim}

\theoremstyle{remark}
\newtheorem{exmp}[thm]{Example}
\newtheorem{rem}[thm]{Remark}

\theoremstyle{definition}
\newtheorem{defn}[thm]{Definition}
\newtheorem{prob}{Problem}[section]

\begin{document}
\title[Subexponential growth and $C^1$ actions on one-manifolds]\thetitle
\date{}


\author[Kim S.-h.]{Sang-hyun Kim}
\address{School of Mathematics, Korea Institute for Advanced Study (KIAS), Seoul, Korea}
\email{kimsh@kias.re.kr}
\urladdr{https://kimsh.kr}

\author[Matte Bon, N.]{Nicolás Matte Bon}
\address{CNRS \&
	Institut Camille Jordan (ICJ, UMR CNRS 5208)\\
	Universit\'e de Lyon\\
	43 blvd.\ du 11 novembre 1918,	69622 Villeurbanne,	France}
\email{mattebon@math.univ-lyon1.fr}
\urladdr{}

\author[de la Salle, M.]{Mikael de la Salle}
\address{CNRS \&
	Institut Camille Jordan (ICJ, UMR CNRS 5208)\\
	Universit\'e de Lyon\\
	43 blvd.\ du 11 novembre 1918,	69622 Villeurbanne,	France}
\email{delasalle@math.univ-lyon1.fr}
\urladdr{}

\author[Triestino, M.]{Michele Triestino}
\address{Institut de Math\'ematiques de Bourgogne (IMB, UMR CNRS 5584)  \& Institut Universitaire de France\\
	Universit\'e de Bourgogne\\
	9 av.~Alain Savary, 21000 Dijon, France}
\email{michele.triestino@u-bourgogne.fr}
\urladdr{}

\begin{abstract}
Let $G$ be a countable group with no finitely generated subgroup of exponential growth. We show that every action of $G$ on a countable set preserving a linear (respectively, circular) order can be realised as the  restriction of some action by $C^1$ diffeomorphisms on an interval (respectively, the circle)  to an  invariant subset.   As a consequence, every action of $G$ by homeomorphisms on a  compact connected one-manifold can be made $C^1$ upon passing to a semi-conjugate action.  The proof is based on a functional characterisation of groups of local subexponential growth.
\end{abstract}

\maketitle
\section{Introduction}

The study of groups of homeomorphisms and diffeomorphisms of connected one-manifolds (that is, real intervals and the circle) is a well-developed topic, and we refer to the classical monographs by Ghys \cite{Ghys2001} and Navas \cite{navas-book} for a detailed account.  A central question in this setting is the relation between the regularity of the action and its combinatorial properties, as well as the intrinsic properties of the acting group. The study of  actions on intervals by homeomorphisms is essentially equivalent to that of actions on sets that preserve linear orders; in particular, a countable group acts faithfully on an interval by homeomorphisms if and only if it is left-orderable, i.e.\ carries a left-invariant linear order, see \cite[Theorem 6.8]{Ghys2001}. For actions on the circle, the corresponding equivalence is true upon considering  \emph{circular} orders rather than linear ones (see Section \ref{s-C1-real} for the definition). 
Moreover, one can always conjugate any action of such a group on a one-manifold to make it by \emph{bi-Lipschitz} homeomorphisms; see 
\cite[Proposition 2.3.15]{navas-book}. 

This purely combinatorial point of view is no longer applicable to actions by $C^r$ diffeomorphisms, for $r\ge 1$. As $r$ increases, such actions satisfy additional restrictions that constrain both the combinatorics of the underlying order-preserving action and the nature of the groups that admit such an action. See the book of the first author and Koberda \cite{KK2021book} for a survey on some recent development. 
In this note we are interested in actions by diffeomorphisms of  class $C^1$. We show that for a group of (local) subexponential growth, all invariant (linear or circular) orders on it are witnessed by $C^1$ actions on one-manifolds. Namely, we show the following. 
\begin{thm} \label{t-C1-growth}
    Let $G$ be a countable group with no finitely generated subgroup of exponential growth. Let $G\curvearrowright X$ be an action on a countable set $X$ which preserves a linear (respectively, circular) order on $X$. Then we have an action $G\curvearrowright M$ on $M=[0, 1]$ (respectively, $M=\mathbb{S}^1$) by $C^1$ diffeomorphisms, and a $G$-equivariant, injective, order-preserving map $i\colon X\hookrightarrow M$.
\end{thm}

In particular, if a countable group $G$ as in the statement is left-orderable, then it embeds in $\Diff^1_+([0, 1])$.
Actually, the proof of Theorem \ref{t-C1-growth}  implies the following:
\begin{cor}\label{c-C1-growth}
If $G$ is a countable group with no finitely generated subgroup of exponential growth, 
then every continuous action $\rho$ of $G$ on a compact connected one-manifold admits a $C^1$ action $\hat\rho$ that is semi-conjugate to $\rho$.
\end{cor}
More precisely, $\hat\rho$ can be obtained from $\rho$ by blowing up any countable dense set of orbits; see Section~\ref{s-C1-real} for precise definitions of semi-conjugacy and blow-ups. 
   It appears to be an open question whether $\rho$ can always be (topologically) \emph{conjugate} to a $C^1$ action.

In fact, Theorem \ref{t-C1-growth} and Corollary \ref{c-C1-growth} are proven in a more general form  (Theorem \ref{t-C1-precise} and Corollary \ref{c-C1-precise}) which does not require subexponential growth of the group, but only of the orbits of the action (when restricted to any finitely generated subgroup). Moreover, the resulting $C^1$ action can be conjugate to be arbitrarily close to the trivial action (in the case of the interval) or to an action by rotations (in the case of the circle), in the $C^1$ topology.

\subsection*{Context and previous work}
Various special cases of Theorem \ref{t-C1-growth} (dealing with specific groups and actions) appear in the literature. For finitely generated abelian groups, the construction of $C^1$ actions on one-manifolds with prescribed dynamics is a classical and well-understood problem, going back to the examples of Denjoy \cite{Denjoy1932}, Pixton \cite{Pixton}, and Tsuboi \cite{Tsuboi1995}  (also dealing with intermediate regularity $C^{1+\alpha}$, i.e.\ $C^1$ with $\alpha$-H\"older derivatives). The constructions in these papers have been influential for most of the works on this particular problem.

In \cite{FF2003}, Farb and Franks showed that every finitely generated, (residually) torsion free nilpotent group admits a faithful action  on a compact interval by $C^1$ diffeomorphisms. They do so by constructing a $C^1$ realisation on $[0,1]$ of an explicit order-preserving action of groups of unipotent matrices with integer coefficients\footnote{There was a small mistake in the published version of \cite{FF2003}  (see Remark 2.1 in the 2018 arXiv version); see also Jorquera \cite{Jorquera2012} for a correct proof.}.
See also \cite{Parkhe2016, CastroJorqueraNavas, JNR2018, EscayolaRivas} for further developments on actions of nilpotent groups in $C^{1+\alpha}$ regularity. Notably, Parkhe shows in \cite{Parkhe2016} that every action of a finitely generated virtually nilpotent group on a one-manifold can be conjugated to a $C^1$ action (in fact $C^{1+\alpha}$ for some $\alpha>0$).
For finitely generated groups of intermediate  (i.e.\ subexponential but faster than polynomial) growth, an example of $C^1$ action was constructed by
Navas \cite{Navas2008GAFA}, who showed that the Grigorchuk--Machì group, defined by Grigorchuk in \cite{Grigorchuk1985}, and realized as a subgroup of $\Homeo_+([0,1])$ by Grigorchuk and Mach\`i in \cite{GrigorchukMachi}, can be embedded into $\Diff^1_+([0, 1])$. His construction is quite indirect, and proceeds by first finding a $C^1$ realisation of an explicit family of non-faithful actions, and then taking a limit.  On the other hand, there are only a few related works discussing smoothability of actions of (countable) groups that cannot be finitely generated: $C^1$ smoothability of certain actions of infinite rank free abelian groups on \emph{non-compact} intervals is shown by Xu, Shi, and Wang in \cite{xu_invariant_2020}, building on work of Bonatti and Farinelli \cite{BonattiFarinelli}, whereas
Jorquera \cite{Jorquera2012} extended the construction of Farb and Franks to a non-finitely generated group of unipotent integer matrices of infinite rank (which is locally nilpotent, but not nilpotent). 

As far as $C^1$ regularity  is concerned, all these realisations results can be seen as special cases of Theorem \ref{t-C1-growth}. The conclusion of Theorem \ref{t-C1-growth} seems to be new in many cases, e.g.\ when the group $G$ is free abelian of infinite rank, or an arbitrary finitely generated group of intermediate growth.
Surprisingly, our proof of Theorem \ref{t-C1-growth} is soft and involves no complicated computation. The main novelty is a characterisation of groups of (local) subexponential growth as those admitting a positive $\ell^1$-function which is asymptotically left-invariant; see Proposition \ref{p-moderate}. With this characterisation in hand, Theorem \ref{t-C1-growth} is proven using the standard technique to construct one-dimensional actions with wandering intervals, employed in most (if not all) the works cited above. This technique is based on regularity properties of well-chosen families of diffeomorphisms (we use the classical Yoccoz' family introduced in \cite{Yoccoz1995}), and the difficulty boils down to find a good choice of lengths for the wandering intervals.  Proposition \ref{p-moderate} allows us to choose these lengths, and the combinatorics of the underlying action plays no role in the proof.

We will only briefly mention that many works (including some of those previously mentioned) also discuss the interesting problem of determining the best regularity that can be achieved in the semi-conjugacy class of a given action on a one-manifold. This so-called \emph{critical regularity} can be determined for finitely generated free abelian groups using the results in Deroin, Kleptsyn, and Navas  in \cite{DKN2007}. Determining the critical regularity for actions of nilpotent groups is a more challenging problem, and we refer for this to the recent work of Escayola and Rivas \cite{EscayolaRivas}. For left-orderable groups of intermediate growth, it turns out that the critical regularity is exactly $C^1$, as follows by combining our Corollary \ref{c-C1-growth} with the result from Navas \cite{Navas2008GAFA}, where it is proved that the group $\Diff^{1+\alpha}_+([0,1])$ contains no group of intermediate growth for any $\alpha>0$.
\begin{cor}
    Let $G$ be a finitely generated, left-orderable group of intermediate growth. Then $G$ embeds in $\Diff^1_+([0,1])$ but it does not embed in $\Diff^{r}_+([0,1])$ for any $r> 1$.
\end{cor}
Regularity $C^{1+\alpha}$ is in general not achievable in Theorem \ref{t-C1-growth} even when $G$ is abelian (but not finitely generated), as follows from \cite{DKN2007}. Nevertheless,  in Section \ref{s-modulus} we explain how explicit estimates on the growth of subgroups of $G$ can be translated into a control on the modulus of continuity of derivatives for the resulting $C^1$ action.
 
We would like to point out that  connections between growth of orbits and the regularity of one-dimensional actions appear (implicitly or explicitly) in various works mentioned above, such as \cite{DKN2007,Navas2008GAFA,Parkhe2016, EscayolaRivas} although most often in the  polynomial growth regime (with the important exception of \cite{Navas2008GAFA}). An early appearance of growth of orbits in the context of one-dimensional group actions can be traced back to Plante \cite{Plante1975}.

Finally, we mention that the assumption on the growth of $G$ is essential in Theorem \ref{t-C1-growth}: there are many known counterexamples for actions of groups of exponential growth, and such examples can be found even among metabelian (hence amenable) groups, see Guelman and Liousse \cite{GL-Baumslag} (and the related work by Cantwell and Conlon \cite{CantwellConlon}),  Bonatti, Monteverde, Navas, and Rivas \cite{BMNR2017MZ}, or the recent work of Brum, Rivas, and the second and fourth authors \cite{BMRT-solvable}. 

\subsection*{Acknowledgements} We would like to thank A.~Navas for suggesting to add the last sentence in Theorem \ref{t-C1-precise} (see also Remark \ref{r-C1-close-2}), and M.~Escayola for explaining to us how Parkhe's proof can be fixed (see Remark \ref{r-alpha-holder}).  NMB and MT thank J.~Brum and C.~Rivas for conversations that motivated this project.

\subsection*{Funding}
NMB and MT thank the hospitality of KIAS, where this project has started. MT thanks the hospitality of ICJ, where this project has been finalized.

SHK is supported by Mid-Career Researcher Program (RS-2023-00278510) through the National Research Foundation funded by the government of Korea, and also by a KIAS Individual Grant (MG073601) at Korea Institute for Advanced Study.
MT and NB are partially supported by the project ANR Gromeov (ANR-19-CE40-0007). MT is also partially supported by the CNRS with a semester of “délégation” at IMJ-PRG (UMR
CNRS 7586), and by the EIPHI Graduate School (ANR-17-EURE-0002).

\section{A functional characterisation of subexponential growth}

If $f\colon X\to \R$ is a real-valued function defined on an arbitrary set, we use the notation $\lim_{x\to \infty} f(x)=l$ if $f(x)$ tends to $l$ as $x$ escapes from finite subsets of $X$. A set $X$ endowed with an action of $G$ will be called a $G$-\emph{set} for short.
\begin{defn}
Let $G$ be a group and $X$ be a $G$-set. A \emph{moderate} $\ell^1$-\emph{function} on $X$ is a function $\nu\in \ell^1(X)$ such that $\nu(x)>0$ for every $x\in X$ and such that for every $g$ in $G$, we have $\lim_{x\to \infty} \nu(gx)/\nu(x)=1$. 
\end{defn}
We say that a $G$-set $X$ has \emph{subexponential growth of orbits} if for every finite subset $S\subset G$ and every point $x\in X$, the function $n\mapsto |S^n x|$ grows subexponentially, i.e. 
\[\lim_{n\to \infty} |S^nx|^{\frac{1}{n}}=1.\]
(We do not require any uniform control on the convergence when $x$ or $S$ vary.) We will prove the following. 
\begin{prop} \label{p-moderate}
    Let $G$ be a countable group. A countable $G$-set $X$ admits a moderate $\ell^1$-function if and only if it has subexponential growth of orbits.

\end{prop}
Clearly the left-translation action of a group $G$ on itself has subexponential growth of orbits if and only if every finitely generated subgroup of $G$ has subexponential growth. Hence  Proposition \ref{p-moderate} implies the following.
\begin{cor}\label{c-moderate}
For a countable group $G$, the following are equivalent.
\begin{enumerate}
\item \label{i-moderate-1} $G$ has no finitely generated subgroup of exponential growth;
\item \label{i-moderate-2} there exists a moderate $\ell^1$-function on $G$;
\item \label{i-moderate-3} for every countable $G$-set $X$, there exists a moderate $\ell^1$-function on $X$.
    
\end{enumerate}
\end{cor}

For the proof  of Proposition \ref{p-moderate}, we need two elementary lemmas. The first one generalizes \cite[Lemma 2.3]{KK2020DCDS}.
\begin{lem} \label{l-subexpo-elementary}
Let $f\colon \N\to [1, +\infty)$ be a positive non-decreasing function of subexponential growth: 
\[\lim_{n\to \infty} f(n)^{\frac{1}{n}}=1.\]
Then there exists a non-decreasing function  $F\colon \N\to [1, +\infty)$ such that $F\ge f$ and $\lim_{n\to \infty}F(n+1)/F(n)=1$. 
\end{lem}
\begin{proof}
For each $n\in\bN$, let us define
\[
F(n):=\max_{m\ge n}
f(m)^{n/m}.\]
Note that 
$F(n)\ge f(n)\ge1$
and that 
\begin{align*}
{F(n)}/{F(n+1)}
&=
\max
\left( f(n), F(n+1)^{n/(n+1)}\right)/F(n+1)
\\
&=
\max
\left( {f(n)}/{F(n+1)}, F(n+1)^{-1/(n+1)}\right)\le 1.\end{align*}
From the condition $\lim_{n\to\infty}F(n+1)^{-1/(n+1)}=1$, 
we conclude that
\[\lim_{n\to\infty} F(n)/F(n+1)=1.\qedhere\]
\end{proof}
In the next lemma, we reinterpret subexponential growth of orbits in terms of group actions on metric spaces by transformations of bounded displacement (also called \emph{wobbling} transformations, see Ceccherini, Grigorchuk, and de la Harpe \cite{CGdH}).

For a distance $d$ on a set $X$, we denote by $B_{(X, d)}(x; n)$ the ball of radius $n$ around $x\in X$; we also write $B(x;n)$ when the meaning is clear. When $X=G$ is a finitely generated group and $d$ is the word metric associated to a finite symmetric generating set $S$, we denote this also by $B_{(G, S)}(g; n)$.
We shall say that a distance $d$ on a set $X$ has \emph{subexponential volume growth} if all balls are finite and for any $x\in X$, the function $n\mapsto |B_{(X, d)}(x; n)|$ has subexponential growth. This condition does not depend on  the choice of $x$.

\begin{lem}\label{l-wobbling}
    Let $G$ be a countable group and $X$ be a countable $G$-set with subexponential growth of orbits. Then there exists a distance $d\colon X\times X\to \N$ with subexponential volume growth such that the action of $G$ on $(X, d)$ has bounded displacement in the following sense:
    \[\sup_{x\in X} d(x, gx)<\infty \,  \text{ for every } g\in G.\] 
\end{lem}
\begin{proof}
We may assume $G$ and $X$ are infinite. 
As a warm-up, let us first consider the case that $G$ is finitely generated, and that its action on $X$ is transitive. Pick a  finite generating set $S$ for $G$, and endow $X$ with the corresponding (unoriented) Schreier graph structure;
namely, we connect two points $x, y\in X$ by an unoriented edge if $y=sx$ for some $s\in S$. 
Then the resulting path metric $d$ on the connected graph $X$ satisfies the desired conclusion.

As the second case, we still suppose that the action of $G$ on $X$ is transitive, but $G$ is not necessarily finitely generated. We fix a non-redundant enumeration
\[
G=\{g_0:=1,g_1,g_2,\ldots\},\]
such that $g_i^{-1}\in\{g_{i-1},g_i,g_{i+1}\}$ for each $i\ge 1$.
Let us choose a map $\lambda\co G\to \bZ_{\ge 0}$ such that
$\lambda(g_0)=0$ and $\lambda(g_i)>0$ for $i\ge 1$, and such that $(\lambda(g_i))$ grows to infinity.
We require that $\lambda(g)=\lambda(g^{-1})$ for all $g$; we will impose more conditions on $\lambda$ later.
We endow $X$ with a structure of metric graph (i.e.\ a graph where each edge is assigned some positive length), by connecting each $x\in X$ to $gx$ by an edge of length $\lambda(g)$ for every $g\in G$.
Let $d$ be the resulting path metric on the connected graph $X$, so that
\[
d(x,y)=\min\left\{\sum_{i=1}^k\lambda(s_i)\;\middle\vert\;
k\ge1\text{ and }y=s_k\cdots s_1x\text{ for some }s_j\in G\right\}.\]
We claim that, for some suitable choice $\lambda$ of the lengths, the distance $d$ has subexponential volume growth. 
To see this, let us fix a basepoint $x_0\in X$.
If $y\in B(x_0;n)$, then we can write
\[ y = s_k\cdots s_1x_0\]
for some $s_1,\ldots,s_k\in G$ such that $\sum_{j=1}^k \lambda(s_j)\le n$. Note that $k\le n$ and that $\lambda(s_j)\le n$ for each $j$.
Setting
 $S(m):=\{g_0,\ldots,g_m\}$, and 
\[
\mu(n):=\max\left\{m\mid \lambda\left(S(m)\right)\sse[0,n]\right\}\in\bZ_+,\]
we have that \[B(x_0; n):=B_{(X,d)}(x_0;n)\sse S(\mu(n))^n.x_0.\]
The assumption that $X$ has subexponentially growing orbits guarantees that for any fixed $m$, we have $\lim_{n\to \infty}|S(m)^n.x_0|^{\frac 1n}=1.$ By choosing the sequence $(\lambda(g_i))$ to grow sufficiently fast, we can ensure that the sequence $(\mu(i))$ grows as slowly as needed. By a diagonal extraction argument, we can arrange choices so that 
\[|B(x_0; n)|^{\frac 1n}\leq \left|S({\mu(n)})^n.x_0\right|^{\frac 1n}\to 1 \text{ as } n\to \infty. \]
It follows that $d$ has subexponential volume growth. 
Since $d(x,gx)\le\lambda(g)$ for all $g\in G$ and $x\in X$, the action of $G$ on $(X,d)$ has bounded displacement.

Finally, we consider the case where the $G$-action on $X$ is not necessarily transitive. Let $X_0, X_1,\ldots$ be an enumeration of the $G$-orbits, and choose basepoints $x_i\in X_i$. For each fixed $m$, we have that
\[
\lim_{n\to\infty}\max_{0\le i\le m} \left| S(m)^n.x_i\right|^{\frac1n}=1.\]
Again, we can choose a fast increasing $\lambda\co G\to \bZ_{\ge 0}$ so that the function $\mu(n)$ defined above satisfies
\[
\lim_{n\to\infty}\max_{0\le i\le \mu(n)} \left| S(\mu(n))^n.x_i\right|^{\frac1n}=1.\]
Let us now endow each $X_i$ with a metric graph structure using $\lambda$ as above, corresponding to a path metric $d_i$ of subexponential volume growth.
To turn $X$ into a connected graph, connect each $x_{i-1}$ to $x_{i}$ by an additional edge of length $\ell_{i}\in\bZ_+$, where the non-decreasing sequence $(\ell_i)$ will be specified shortly. Let $d$ be the resulting distance on $X$. 
We have
\begin{equation}\label{e-ball}
\left|B_{(X, d)}(x_0; n)\right|
\le
\sum_{j=0}^{m(n)} \left|B_{(X_j, d_j)}(x_j; n)\right|,
\end{equation}
where $m(n)$ is the largest integer $m$ such that
 $\ell_1+\cdots+\ell_m\le n$.
For any fixed $m$, the function 
$
n\mapsto 
\sum_{j=0}^{m}\left|B_{(X_j, d_j)}(x_j; n)\right|$
has subexponential growth. As in the previous case,  a diagonal extraction argument allows us to choose the lengths $(\ell_i)$ so that the right-hand side in \eqref{e-ball} has subexponential growth. Again as in the previous case, we have $d(x, gx)\leq \lambda(g)$ for every $x\in X$ and $g\in G$.
\end{proof}

\begin{proof}[Proof of Proposition \ref{p-moderate}] Let $G$ be a countable group. Suppose that $X$ is a countable $G$-set admitting a moderate $\ell^1$-function $\nu\in \ell^1(X)$, and let us show that $X$ must have subexponential growth of orbits. Fix $x\in X$ and a finite subset $S\subset G$. Upon enlarging $S$, we can assume that it is symmetric and contains the identity. Let $H$ be the subgroup generated by $S$, and $Y\subset X$ the $H$-orbit of $x$. 
Let $d$ be the Schreier graph distance on $Y$ associated with the generating set $S$, so that $B_{(Y, d)}(x; n)=S^nx$. Fix $\lambda>1$. By definition of moderate $\ell^1$-function, we can find $m\in \N$ such that for every $s\in S$ and every $y\in Y\setminus B_{(Y, d)}(x; m)$, we have $\lambda^{-1}\leq \nu(sy)/\nu(y)\leq \lambda$. Set $c=\min\{\nu(y): y\in B_{(Y, d)}(x; m)\}$. Fix now $n\ge m$ and $y\in B_{(Y, d)}(x; n)$.  We can choose a geodesic path of minimal length connecting $y$ to a point in $B_{(Y, d)}(x; m)$, and denote by $y_1$ the endpoint of this path (we take $y_1=y$ if $y\in B_{(Y, d)}(x; m)$). Since the length of such a path does not exceed $n$, we have $\nu(y)\ge \lambda^{-n}\nu(y_1)\ge \lambda^{-n}c$.
It follows that 
\[\lVert \nu\rVert_{\ell^1(X)}\ge \sum_{y\in B_{(Y, d)}(x; n)}\nu(y)\ge \lambda^{-n}c|B_{(Y, d)}(x; n)|.\]
Hence $|S^nx|=|B_{(Y, d)}(x; n)|\leq c^{-1}\lVert \nu\rVert_{\ell^1(G)}\lambda^n$ and $\limsup_{n\to \infty}|S^nx|^{\frac{1}{n}}\leq \lambda$. Since $\lambda>1$ and $S, x$ were arbitrary, it follows that $X$ has subexponential growth of orbits.

Conversely, suppose that $X$ has subexponential growth of orbits, and let us construct a moderate $\ell^1$-function. Let $d$ be a distance on $X$ as in Lemma \ref{l-wobbling}. Fix a basepoint $x_0\in X$. In what follows we write $B(n)$ for the ball $B_{(X, d)}(x_0; n)$. The function  $n\mapsto |B(n)|$ has subexponential growth, and hence so does $n\mapsto  n^2|B(n)|$. Applying Lemma \ref{l-subexpo-elementary}, we can find a non-decreasing function $F(n)\ge n^2|B(n)|$ such that $F(n+1)/F(n)\to 1$. Define a function $\nu\colon X\to \R$ by setting $\nu(x)=F(n)^{-1}$ for every $x\in \partial B(n):= B(n)\setminus B(n-1)$. Note that $\nu\in \ell^1(X)$, since
\[\sum_{x\in X}\nu(x)=\sum_n F(n)^{-1}|\partial B(n)|\leq \sum_n F(n)^{-1}|B(n)|\leq \sum_n \frac{1}{n^2}.\]
Fix now $g\in G$, and let $c=\sup_{x\in X} d(x, gx)$. Given  $\lambda>1$,  there exists $m$ such that $F(n+1)/F(n)\leq \lambda^{\frac{1}{c}}=:\lambda_1$ for $n\ge m$. 
Now for every $x\notin B(m+c)$, with $n=d(x_0, x)$, we have $d(x_0, gx)=n_1$ for some $n_1\in [n-c, n+c]$ (in particular $n_1\ge m$). Hence $\nu(gx)/\nu(x)=F(n)/F(n_1)$ can be written as a telescopic product of at most $c$ terms of the form $F(i\pm 1)/F(i)$ with $i\ge m$, and for each such $i$ we have $\lambda_1^{-1}\leq F(i\pm 1)/F(i)\leq \lambda_1$. It follows that
\[\lambda^{-1}=\lambda_1^{-c} \leq \nu(gx)/\nu(x)\leq  \lambda_1^c=\lambda\]
for every $x\notin B(m+c)$. Since $\lambda>1$ was arbitrary, this shows that $\nu$ is a moderate $\ell^1$-function. \qedhere

\end{proof}
For later use, we record an additional consequence of the proof of Proposition \ref{p-moderate} in the following remark.
\begin{rem} \label{r-C1-close}
  Suppose that $X$ is a countable $G$-set with subexponential growth of orbits. Then for any finite subset $S\subset G$ and any $\delta>0$, we can find  a moderate $\ell^1$-function $\nu$ on $X$ such that 
  \[\left\lvert\frac{\nu(sx)}{\nu(x)}-1\right\rvert<\delta\]
  for every $s\in S$ and every $x\in X$.  Indeed, following the proof of Proposition \ref{p-moderate}, let $c\in \N$ be such that $d(sx, x)<c$ for every $s\in S$ and $x\in X$, and let $m>0$ be such that $\frac{F(n+c)}{F(n)}\leq 1+\delta$ for every $n\ge m$. Then we modify the definition of the function $\nu$,  by setting $\nu\equiv F(m)^{-1}$ on the whole ball $B(m)$, and $\nu\equiv F(n)^{-1}$ on $\partial B(n)$ for $n>m$, to ensure the desired  conclusion.

\end{rem}
\section{$C^1$ realisation of order-preserving actions}\label{s-C1-real}
We now prove Theorem \ref{t-C1-growth} and Corollary \ref{c-C1-growth}, in slightly more precise forms. 
First, let us clarify some of the terminology employed in the statements.
By a \emph{linear order} on a set, we simply mean a total order on it. 
A \emph{circular} (or \emph{cyclic}) order $c$ on a set $X$  is a function $c\colon X\times X\times X\to \{0,\pm 1\}$, satisfying the following conditions: for any $x_1,x_2,x_3,x_4\in X$ one has
\begin{itemize}
    \item $c(x_1,x_2,x_3)=0$ if and only if $x_i=x_j$ for some $i\neq j$,
    \item $c(x_2,x_3,x_4)-c(x_1,x_3,x_4)+c(x_1,x_2,x_4)-c(x_1,x_2,x_3)=0$.
\end{itemize}
A map $f\colon (X, c_x)\to (Y, c_Y)$ is \emph{cyclic-order-preserving}  if
\[c_Y(f(x_1), f(x_2), f(x_3))=c_X(x_1, x_2, x_3)\]
for every $x_1, x_2, x_3$.  A cyclically ordered $G$-set is by definition a cyclically ordered set $(X, c)$ on which a group $G$ acts by cyclic-order-preserving transformations.

A cyclic order $c$ on $X$ induces linear order $\prec$ on the complement $X\setminus\{x_\ast\}$ of any point $x_\ast\in X$, defined by $x\prec y$ if $c(x, y, x_\ast)=1$. Conversely any linearly ordered set $(X,\prec)$ can be turned into a circularly ordered set $(\widehat{X},c)$ by declaring $\widehat{X}=X\sqcup\{x_*\}$ and $c(x,y,x_*)=1$ if and only if $x<y$ (and then use the cocycle relation above to extend it to any triple). Moreover, if a group $G$ acts by order-preserving automorphisms of $(X,\prec)$, we can extend it to a cyclic-order-preserving action of $G$ on $(\widehat{X},c)$, by letting $x_*$ be a global fixed point. Because of this, in the proof of Theorem \ref{t-C1-growth} it is actually enough to consider cyclic-order-preserving action.



Finally, we denote by $\Diff^1_+(M)$ the group of $C^1$ diffeomorphisms of a one-manifold. The derivative of a diffeomorphism $f$ is denoted by $Df$. The following result implies Theorem \ref{t-C1-growth}.
\begin{thm}\label{t-C1-precise}
Let $G$ be a countable group, and $(X, c)$ be a cyclically ordered $G$-set with subexponential growth of orbits. Then there exist an action $\rho\colon G\to \Diff^1_+(\mathbb{S}^1)$ and a family of open intervals $(J_x)_{x\in X}$ in $\mathbb{S}^1$  such that:    
\begin{enumerate}[label=(\roman*)]
\item \label{t-i-1}if $x\neq y$, then $J_x\cap J_y=\varnothing$, and $\bigcup_{x\in X} J_x$ is dense in $\mathbb{S}^1$; 
\item \label{t-i-2} the map $X\to \mathbb{S}^1$ that associates to $x$ the left  (equivalently, the right) endpoint of $J_x$ is cyclic-order-preserving; 
\item \label{t-i-3} we have $\rho(g)(J_x)=J_{gx}$ for every $x\in X$, and if $gx=x$ then $\rho(g)$ is the identity on $J_x$.
\end{enumerate}
Furthermore, given any finite subset $S\subset G$ and any $\varepsilon>0$, the action $\rho$ above can be chosen such that  $\sup_{\xi\in \mathbb{S}^1}|D\rho(s)(\xi)-1|<\varepsilon$ for every $s\in S$.

\end{thm}
\begin{proof}
Let us start by fixing a family $\{\phi^I_J\co I\to J\}$ of $C^2$ diffeomorphisms defined for each pair of  nondegenerate compact intervals $I$ and $J$ in $\bR$. We require this family to satisfy the following, where $D$ denotes the derivative.
\begin{enumerate}[label=(C\arabic*)]
    \item\label{i1-eqfam} $|D\phi^I_J(x)-1|\le ||J|^2/|I|^2-1|$ for $x\in I$;
    \item\label{i2-eqfam} $D\phi^I_J(x)=1$ for $x\in\partial I$;
    \item\label{i3-eqfam} for three compact intervals $I,J,K$ in $\bR$ we have
    \[ \phi^J_K\circ \phi^I_J=\phi^I_K.\]
\end{enumerate}
For instance, one can employ Yoccoz' equivariant family (\cite{Yoccoz1995}, cf.\ \cite[Section 4.2]{KK2021book}).

    By Proposition \ref{p-moderate}, we can fix a moderate   $\ell^1$-function $\nu$ on $X$. We renormalize $\nu$ to assume that $\lVert \nu\rVert_{\ell^1}=1$. Fix a basepoint $x_\ast\in X$, and let $\prec$ be the linear order induced by $c$ on $X\setminus \{x_\ast\}$. We define two functions $j_-, j_+\colon X\to \mathbb{S}^1=\R/\Z$, as follows. For $x\in X\setminus \{x_\ast\}$  we set
    \[j_-(x)=\sum_{\substack{y\in X\setminus\{x_\ast\} \\ y\prec x}} \nu(y), \quad j_+(x)=\sum_{\substack{y\in X\setminus\{x_\ast\} \\ y\preceq x}} \nu(y) \quad \pmod{\Z}.\]
 Finally we set $j_-(x_\ast)= 1-\nu(x_\ast)$ and $j_+(x_\ast)=0\pmod{\Z}$.
Note that both functions are injective and cyclic-order-preserving. Let $J_x\subset (0, 1)$ be the open interval $(j_-(x), j_+(x))$. By construction, its length is $\lvert J_x\rvert=\nu(x)$. It follows that $\mathcal{O}:=\bigcup_x J_x$ has measure 1 (in particular, it is dense), and it is straightforward to check that the family $(J_x)$ satisfies conditions \ref{t-i-1} and \ref{t-i-2} in the statement. Let $\mathcal{C}=\mathbb{S}^1\setminus\mathcal{O}$. We define the action $\rho$ on the set $\mathcal{O}$, by setting $\rho(g)|_{J_x}=\phi^{J_x}_{Jg(x)}$ for every $g\in G$ and $x\in X$. By condition \ref{i3-eqfam} this is a well-defined, cyclic-order-preserving action on $\mathcal{O}$ by $C^1$ diffeomorphisms, and it extends by continuity to an action on $\mathbb{S}^1$ by homeomorphisms, still denoted $\rho$, which satisfies \ref{t-i-3}. To check that the extension is actually $C^1$, it is enough to check that for every sequence $(\xi_n)\subset \mathcal{O}$ converging to a point of $\zeta\in \mathcal{C}$ we have $D\rho(g)(\xi_n)\to 1$. Let $x_n$ be such that $\xi_n\in J_{x_n}$. By an extraction argument, we can assume that either $(x_n)$ is eventually constant, or that it escapes from every finite subset of $X$. In the first case, convergence of the derivative follows from \ref{i2-eqfam}, and in the second it follows from \ref{i1-eqfam} and from the fact that $\nu$ is a moderate $\ell^1$-function.  

We finally justify the last sentence of the theorem. Fix a finite subset $S\subset G$ and $\varepsilon>0$.  Considering the modified version of $\nu$ as described in Remark \ref{r-C1-close}, with $\delta$ satisfying $\delta(2+\delta)\le \varepsilon$. Conditions \ref{i1-eqfam} and \ref{i2-eqfam} give
    \[
    |D\rho(s)(x)-1|\le \lvert\nu(sx)^2/\nu(x)^2-1\rvert\le \delta(2+\delta)\le \varepsilon.
    \]
    for any $x\in \mathbb{S}^1$. 
\end{proof}

\begin{rem}\label{r-C1-close-2}
    The last sentence in the previous result was suggested to us by Andrés Navas. An equivalent formulation is that the action $\rho$ can be (topologically) conjugate to be made arbitrarily close, in the $C^1$ topology, to a homomorphism to the group of rotations (which is necessarily trivial if $\rho$ has a global fixed point, e.g.\ if we start from a linear-order-preserving action). This has to be compared with the following results of Navas and Parkhe. First, \cite[Théorème A]{Navas2014}, states that any action of the circle of a finitely generated group of subexponential growth can be $C^0$ conjugated to a bi-Lipschitz action such that the Lipschitz constants of the images of an arbitrary finite subset can be chosen arbitrarily close to 1. Second, \cite[Théorème B]{Navas2014} and \cite[Corollary 1.6]{Parkhe2016} (proved independently), state that any $C^1$ action of a finitely generated nilpotent group can be $C^0$ conjugated to a $C^1$ action such that the derivatives of the images of an arbitrary finite subset can be chosen arbitrarily close to 1. The embedding of the Grigorchuk--Machì group in $\Diff_+^1([0, 1])$ constructed by Navas in  \cite{Navas2008GAFA} also has this property.

    Let us also mention that it is an open problem to characterise which $C^1$ actions on the circle can be (topologically) conjugate  arbitrarily close to rotations in the $C^1$ topology. As stated in \cite{Navas2014} and with a slight variation in \cite[Question 4]{BonattiFarinelli}, it is not known if the unique obstruction is the existence of a so-called \emph{resilient orbit} (without entering into the somehow technical definition, let us only comment that any resilient orbit has exponential growth).
\end{rem}

Let $M$ be a compact connected one-manifold, and let $G$ be a group. 
For two representations $\hat\rho, \rho$ in $\Hom(G,\Homeo(M))$, we say $\hat\rho$ is a \emph{blow-up}  of $\rho$  if there exists a continuous,  surjective map $h\co M\to M$ such that
$h^{-1}(J)$ is connected for each connected set $J\sse M$,
and such that each $g\in G$ satisfies
\[
h\circ\hat\rho(g)=\rho(g)\circ h.\]
If $\hat\rho$ is a blow-up of $\rho$, we also say that $\hat\rho$ and $\rho$ are \emph{semi-conjugate} ~\cite{KH1995}. In full generality, semi-conjugacy\footnote{There are some mildly different definitions in the literature. Note however that we do not really need  the general definition of semi-conjugacy in this paper, so that the reader can stick to the special case of blow-ups, which is included in all definitions. } is defined as the equivalence relation between actions on a connected one-manifold \emph{generated by} the pairs $(\hat{\rho}, \rho)$ such that $\hat{\rho}$ is a blow up of $\rho$. (See~\cite{BFH2016} and~\cite[Appendix A]{KKM2019} for  equivalent and more symmetric formulations of semi-conjugacy.)

Corollary \ref{c-C1-growth} is a consequence of the following. 
\begin{cor} \label{c-C1-precise}
If $G$ is a countable group, and if $\rho$ is a topological action of $G$ on a compact connected one-manifold $M$ with subexponential growth of orbits, then there exists a $C^1$ action $\hat\rho$ that is a blow-up of $\rho$.\end{cor}

\begin{proof}
We only consider the case $M=\mathbb{S}^1$,
since  the case of an interval would easily follow by identifying the extreme points.
Let $X\subset \mathbb{S}^1$ be a countable, dense, $\rho(G)$-invariant subset. We apply Theorem \ref{t-C1-precise} to the action of $G$ on $X$, endowed with the circular order induced  by the inclusion $X\subset \mathbb{S}^1$. Let $\widehat{\rho}$ be the resulting $C^1$ action, and $(J_x)_{x\in X}$ be the associated family of intervals. There is a unique continuous $G$-equivariant map  $h\colon \mathbb{S}^1\to \mathbb{S}^1$ that on each interval $J_x$ takes the constant value $x$, which shows that $\widehat{\rho}$ is a blow-up of $\rho$.  \qedhere
\end{proof}

\section{On the modulus of continuity of derivatives} \label{s-modulus}
It is natural to ask whether the previous results can be improved to obtain actions with slightly better regularity than $C^1$. As discussed in the introduction,  regularity $C^{1+\alpha}$ for some $\alpha>0$ cannot always be achieved in the generality considered here (for instance this is impossible for certain actions of infinitely generated abelian groups, as follows from \cite{DKN2007}, and for actions of finitely generated groups of intermediate growth  \cite{Navas2008GAFA}). However, as should not be surprising, in the setting of Theorem \ref{t-C1-precise}  a quantitative control of the growth of orbits of the action on $X$ translates into a control for the modulus of continuity of the derivatives of the resulting $C^1$ action on $\mathbb{S}^1$. Here we briefly review our arguments to explain how to obtain such explicit bounds. This relies on finer regularity properties of Yoccoz' equivariant family, as analysed in \cite[\S 1.3]{Navas2008GAFA} (see also \cite{Jorquera2012}). 

Recall that a \emph{modulus (of continuity)} is a homeomorphism
$\sigma\co [0,\infty)\to[0,\infty)$.
A function  $f$ between metric spaces is in the class $C^{\sigma}$ if
\[
\sup_{x\ne y} |d(f(x),f(y))|/\sigma(d(x,y))\]
is finite. The supremum in the previous expression is then called the $C^\sigma$-norm of $f$. For $M=\mathbb{S}^1$ or $[0,1]$, we denote by $\Diff_+^{1,\sigma}(M)$ the set of all orientation--preserving $C^1$ diffeomorphisms of $M$ whose derivatives are in the class $C^\sigma$. It is a group if the modulus $\sigma$ is not too wild, for example if $t\mapsto \sigma(t)/t$ is decreasing.
\begin{prop}[Quantitative version of Theorem \ref{t-C1-precise}] \label{p-quantitative}
    Retain all assumptions of Theorem \ref{t-C1-precise}, and suppose moreover that  $F\colon \N\to [1, +\infty)$ is a non-decreasing function such that $\lim F(n+1)/F(n)=1$, and that for every finite subset $S\subset G$ and every $x\in X$, we have 
    \begin{equation}
\label{i-summable} 
  \sum_n \frac{|S^nx|}{F(n)}<+\infty.
    \end{equation}
 Assume further that $\sigma$ is a modulus of continuity such that $t \mapsto \sigma(t)/t$ is decreasing, and that
\begin{equation} \label{i-F-omega} 
\limsup_{n\to \infty}  \left\lvert \frac{F(n+1)}{F(n)}-1\right\rvert \cdot \frac{1}{\sigma(F(n)^{-1})}<+\infty.
\end{equation}  
    Then the action $\rho$ in the conclusion of Theorem \ref{t-C1-precise} can be chosen to take values in $\Diff_+^{1, \sigma}(\mathbb{S}^1)$.
  
\end{prop}
\begin{proof}[Proof sketch.] 
Let $X$ be the $G$-set in the statement, and choose a basepoint $x_0\in X$. First, let us look back at the proof of Lemma \ref{l-wobbling}. A careful look at the proof shows that if $f(n)$ is any function such that 
\[\lim_n f(n)/\lvert S^nx\rvert=+\infty\]
holds for every finite $S\subset G$ and every $x\in X$, then the distance $d$ on $X$ in the conclusion of Lemma \ref{l-wobbling} can be chosen such that 
\[\lim_n \left(f(n)/|B_{(X, d)}(x_0; n)|\right)=+\infty.\]
Indeed all the diagonal extraction arguments in the proof can be improved to ensure this. In our situation, thanks to assumption \eqref{i-summable}, there exists such an $f$ with the additional property that $\sum_n F(n)^{-1}f(n)<+\infty$, and so the obtained distance $d$ will satisfy  $\sum F(n)^{-1}|B_{(X, d)}(x_0; n)|<+\infty$.

Now consider the proof of Proposition \ref{p-moderate}. From the previous paragraph, it follows that in order to construct the moderate $\ell^1$-function $\nu$,  we can choose the function $F(n)$ appearing in the proof as the function $F(n)$ given here.  (Note that the assumption $F(n)\ge n^2|B(n)|$ made in the proof was used only to ensure the convergence of $\sum_n F(n)^{-1}|B(n)|$). With this choice, thanks to assumption \eqref{i-F-omega}, the function $\nu\in \ell^1(X)$ constructed in the proof Proposition \ref{p-moderate} will have the property that for every $g\in G$, we have
\begin{equation} M_g:=\sup_{x\in X}\left(\left\lvert\frac{\nu(gx)}{\nu(x)}-1\right\rvert \cdot \frac{1}{\sigma(\nu(x))}\right)<+\infty.\label{e-omega-moderate} \end{equation}
With this estimate in hand, consider the proof of Theorem \ref{t-C1-growth}. Let $(J_x)_{x\in X}$ be the family of intervals constructed there, and recall that $J_x$ has length $\nu(J_x)$. We take the family of $\{\phi^I_J\}$ to be Yoccoz' family. By \cite[Lemma 1.10]{Navas2008GAFA}, the derivative  $D\phi^I_J$ has $C^\sigma$-norm bounded above by
\[6\pi \left(\left \lvert \frac{|J|}{|I|}-1\right \rvert\cdot \frac{1}{\sigma(|I|)}\right)\]
whenever $1/2\leq |I|/|J|\leq 2$. Hence \eqref{e-omega-moderate} implies that $D\phi^{J_{gx}}_{J_x}$ has  $C^\sigma$-norm bounded by $6\pi M_g$ for all but finitely many $x$. Hence \cite[Lemma 1.12]{Navas2008GAFA}  implies that $D\rho(g)$ has finite $C^\sigma$-norm.\qedhere
\end{proof}

We now explain how Proposition \ref{p-quantitative} specializes in a few concrete examples. 
\begin{exmp} \label{e-alpha-holder}
    Suppose that the action of $G$ on $X$ has polynomial growth of orbits with uniformly bounded degree, namely there exists some $d$ such that for every $x\in X$ and $S\subset G$ finite, we have $\sup_n |S^nx|/n^d<+\infty.$ Then one can choose an $\alpha$-H\"older modulus of continuity $\sigma(s)=s^\alpha$. Namely, in Proposition \ref{p-quantitative}, we can take $F(n)=n^c$ with $c>d+1$, and taking $\alpha<\frac{1}{d+1}$ guarantees that \eqref{i-F-omega} is satisfied. 
    
    \end{exmp}
    
    \begin{rem} \label{r-alpha-holder}
        
   In the previous example, the conclusion  can be slightly improved by observing that it is enough to ensure convergence of the series \eqref{i-summable} where balls are replaced by \emph{spheres} (as apparent from the proof of Proposition \ref{p-moderate}). In particular for actions of $\Z^d$ one can choose any $\alpha<\frac{1}{d}$, as was shown in   \cite{DKN2007}. For a general finitely generated nilpotent group of degree of  growth $d$, the same bound $\alpha<\frac{1}{d}$ works and was claimed by Parkhe \cite{Parkhe2016}, but the proof contains a mistake, as it relies on the unproven assumption that spheres have growth $\sim n^{d-1}$ (it is not known whether this is true); however, this can be easily fixed using the doubling property of balls, as kindly explained to us by Maximiliano Escayola and Victor Kleptsyn.

    \end{rem}
%

\begin{exmp}
    Suppose now that the action of $G$ on $X$ has polynomial growth of orbits with non-uniformly bounded degree, namely for every $x\in X$ and finite $S\subset G$, there exists some $d= d(x, S)$ such that  we have $\sup_n |S^nx|/n^d<+\infty$ (this is the case, for instance, if the group $G$ is free abelian of infinite rank, and more generally if it is locally nilpotent but not finitely generated). Then one can choose $\sigma$ to be any given \emph{non-H\"older} modulus of continuity, i.e.\ any function  such that $\lim_{t\to0+}t^\alpha/\sigma(t)=0$ for every $\alpha\in(0, 1)$ (with $\sigma(t)/t$  decreasing). Indeed, given such a $\sigma$, we can apply Proposition \ref{p-quantitative} with $F$ of the form $F(n)=n^{\beta(n)}$, where $\beta\colon [1,\infty)\to [1,\infty)$ is a suitable non-decreasing function.  Condition \eqref{i-summable} is satisfied as soon as $\lim_{x\to \infty}\beta(x)=+\infty$, and $\frac{F(n+1)}{F(n)}$ tends to 1 as soon as $\beta$ is concave and does not increase too fast (for instance $\beta(x)\leq \sqrt{x}$ suffices). To check \eqref{i-F-omega}, note that if $\beta(x)$ grows sufficiently slowly, for every large enough $n$ we have
    \[\left \lvert\frac{F(n+1)}{F(n)}-1\right \rvert\frac{1}{\sigma(F(n)^{-1})}=\left\lvert\left(1+\frac{1}{n}\right )^{\beta(n)}-1\right\rvert\frac{1}{\sigma(n^{-\beta(n)})}\leq \frac{2\beta(n)}{n\sigma(n^{-\beta(n)})}.\]
Here we used the elementary inequality $(1+x)^c\leq 1+2c x$, holding for every $c\geq 0$ and every $x$ in some interval of the form $(0, \varepsilon_c)$, and assumed that $\beta(n)$  grows slowly enough so that $1/n\leq \varepsilon_{\beta(n)}$ for every large enough $n$. Since $\sigma$ is non-H\"older, for any fixed $\beta>0$ we have $\lim_n \frac{\beta}{n\sigma(n^{-\beta})}=0$  (as shown by the change of variable $s=n^{-\beta}$).  Hence if we choose $\beta(n)$ to grow slowly enough, the right-hand side in the previous expression tends to 0.
    \end{exmp}

 \begin{exmp}

 
     Suppose that the action of $G$ on $X$ has stretched exponential growth of degree at most $\alpha<1$, that is $\sup_n |S^n x|/\exp(n^\alpha)<\infty$ for every $x \in X$ and $S\subset G$ finite. Then the action of $G$ in the conclusion of Theorem \ref{t-C1-precise} can be chosen to have regularity $C^{1,\sigma}$ with modulus of continuity given by $\sigma(s) = \lvert \log s\rvert ^{1-1/\alpha}$ for small $s$. Indeed, in Proposition \ref{p-quantitative}, we can take $F(n)=n^2 \exp(n^\alpha)$. Condition~\eqref{i-summable} is obvious, Condition~\eqref{i-F-omega} holds because
     \[\frac{F(n+1)}{F(n)} - 1 = O(n^{\alpha -1}) = O(\sigma(F(n)^{-1})),\] and $t\mapsto \sigma(t)/t$
     is decreasing on $(0,\exp(1-1/\alpha))$.
 \end{exmp}

\bibliographystyle{amsplain}
\bibliography{biblio_michele}
\end{document}